\documentclass{amsart}

\usepackage[english]{babel}

\usepackage{mathrsfs, mathtools, amssymb}

\usepackage{dsfont}

\usepackage[paper=a4paper, margin=3cm]{geometry}

\usepackage[breaklinks]{hyperref}       
\usepackage{url}            
\usepackage{booktabs}       
\usepackage{amsfonts}       
\usepackage{nicefrac}       
\usepackage{microtype}      
\usepackage{lipsum}

\usepackage{amsmath,amssymb,amsthm,amscd}
\usepackage{amsrefs}
\usepackage{tikz,tikz-cd}
\usetikzlibrary{arrows,arrows.meta}

\usepackage{graphicx}
\graphicspath{}
\DeclareGraphicsExtensions{.pdf,.png,.jpg,.jpeg}

\usepackage{yhmath}
\usepackage{enumitem}
\usepackage{breakurl}

\newcounter{stmcounter}[section]

\newcounter{thmMaincounter}

\newtheorem{formula}{}[section]
\newtheorem{proposition}[formula]{Proposition}
\newtheorem{corollary}[formula]{Corollary}
\newtheorem{lemma}[formula]{Lemma}
\newtheorem{theorem}[formula]{Theorem}
\newtheorem{theoremM}[thmMaincounter]{Theorem}

\theoremstyle{definition}

\newtheorem{definition}[formula]{Definition}
\newtheorem{example}[formula]{Example}

\theoremstyle{remark}

\newtheorem{remark}[formula]{Remark}

\newcommand{\sk}{\mathrm{sk}}

\newcommand{\Z}{\mathbb Z}

\newcommand{\C}{\mathbb C}
\newcommand{\Q}{\mathbb Q}
\newcommand{\N}{\mathbb N}

\newcommand{\lb}{\lbrace}
\newcommand{\rb}{\rbrace}

\newcommand{\lv}{\lvert}
\newcommand{\rv}{\rvert}
\newcommand{\mc}{\mathcal}

\DeclareMathOperator{\rk}{rk}

\DeclareMathOperator{\ord}{ord}

\DeclareMathOperator{\Tor}{Tor}

\DeclareMathOperator{\Id}{Id}

\DeclareMathOperator{\hoc}{hocolim}

\DeclareMathOperator{\Hm}{Hom}

\DeclareMathOperator{\lm}{lim}
\DeclareMathOperator{\lk}{lk}
\DeclareMathOperator{\srep}{srep}

\DeclareMathOperator{\lmm}{lim}

\DeclareMathOperator{\tr}{tr}
\DeclareMathOperator{\Sk}{Sk}
\DeclareMathOperator{\cosk}{cosk}

\DeclareMathOperator{\Cone}{Cone}

\DeclareMathOperator{\Top}{Top}

\renewcommand{\leq}{\leqslant}
\renewcommand{\geq}{\geqslant}

\renewcommand{\phi}{\varphi}
\newcommand{\ot}{\leftarrow}

\newcommand{\Cc}{\mathbf{C}}
\newcommand{\Pp}{\mathbf{P}}
\newcommand{\sset}{\mathbf{sSet}}
\newcommand{\stp}{\mathbf{sTop}}
\newcommand{\Mod}{\mathbf{Mod}}
\renewcommand{\Top}{\mathbf{Top}}

\tikzcdset{arrow style=tikz, diagrams={>=stealth}}

\newcommand*\pFqskip{8mu}
\catcode`,\active
\newcommand*\pFq{\begingroup
	\catcode`\,\active
	\def ,{\mskip\pFqskip\relax}%
	\dopFq
}
\catcode`\,12
\def\dopFq#1#2#3#4#5{%
	{}_{#1}F_{#2}\biggl[\genfrac..{0pt}{}{#3}{#4};#5\biggr]%
	\endgroup
}

\begin{document}
	
	\title[Rational cohomology of toric diagrams]{Rational cohomology of toric diagrams}
	
	\author{Grigory Solomadin}
	\address[G.\,Solomadin]{Philipps-Universit\"at Marburg, Germany}
	\email{grigory.solomadin@gmail.com}
	
	
	\begin{abstract}
		In this note, (rational) Betti numbers of homotopy colimits for toric diagrams and their classifying spaces are described in terms of sheaf cohomology over CW posets.
		We prove for any~$T$-diagram~$D$ over any CW poset that Cohen-Macaulayness (over~$\Q$) of the~$T$-action on~$\hoc D$ is equivalent to acyclicity for a certain sheaf.
		The ordinary and bigraded Betti numbers are computed for skeletons of equivariantly formal spaces from this class (in particular, of compact smooth toric manifolds).
	\end{abstract}
	
	\date{\today}
	
	\maketitle
	
	\section{Introduction}
	
	A diagram~$D\colon \Cc\to\Top$ over a small category~$\Cc$ is a functor in the category~$\Top$ of CGWH topological spaces and continuous maps between them.
	For any diagram~$D$ the spectral sequence of the skeletal filtration on the nerve of~$\Cc$ (in singular cohomology with any PID~$\Bbbk$ coefficients) converges to~$H^*(\hoc D;\Bbbk)$~\cite{bo-ka-72, bo-87}.
	Assume that~$\Cc=\Pp$ is a finite graded poset.
	As it is well known, this spectral sequence is isomorphic to the Mayer-Vietoris spectral sequence (MVSS, for short) for the open cover of~$\hoc D$ induced by the open stars of~$\Pp$, e.g. see~\cite[\S5]{so-to-26}.
	The second page of this spectral sequence is given by the respective cohomology groups of the sheaf~$H^*(D)$ on the poset~$\Pp^{op}$, or equivalently by right derived functors of~$\lm H^*(D)$.
	If~$\Pp$ is a CW poset~\cite{bj-84} (equivalently,~$\Pp^{op}$ is an orientable homology manifold) then the sheaf (co)homology groups are isomorphic to the respective cellular (co)homology, provided that the sheaf stalks are free~$\Bbbk$-modules~\cite{ev-tu-15}.
	A diagram is called a \textit{toric diagram} if it is a functor in the category whose objects are compact tori (i.e.~$(S^{1})^{n}$,~$n\in\Z_{\geq 0}$) and whose arrows are group homomorphisms~\cite{we-zi-zi-99}.
	A toric diagram~$D$ is called a \textit{$T$-diagram} if there is an object-wise epimorphism~$\kappa T\to D$ of diagrams, where~$\kappa T$ denotes the constant diagram.
	For any toric diagram, as well as its composition with the classifying functor, MVSS (over~$\Q$) abuts at its second page by natural formality over~$\Q$ of Eilenberg-MacLane spaces established in~\cite{so-to-26}.
	
	Recall that a normal, reduced, irreducible variety over~$\C$ with an effective action of an algebraic torus with a dense open orbit is called a \textit{toric variety} (e.g. see~\cite{cls-11} for overview of the theory).
	The underlying complex manifold of a compact nonsingular toric variety is a \textit{toric manifold}~\cite{da-ja-91} with respect to the compact form of the algebraic torus.
	Let~$X_{\Sigma}$ be any toric variety with the corresponding rational polyhedral fan~$\Sigma$ of dimension~$n$.
	By~\cite{fr-10}, there is the~$T$-diagram~$D_{\Sigma}$ such that~$X_{\Sigma}$ is~$T$-equivariantly homotopy equivalent to~$\hoc D_{\Sigma}$.
	Here~$\Sigma$ denotes (by slightly abusing the notation) the poset of cones in the fan~$\Sigma$ ordered by inclusion.
	
	Denote by~$H^{*}(X;\Q)$ the singular cohomology groups (over~$\Q$) for a topological space~$X$, and by~$H^{cw}_{i}(\Sigma;\mathcal{F})$ the cellular homology with coefficients in a cellular cosheaf~$\mathcal{F}$ (see~\S\ref{sec:shposet}).	
	Below we outline three main results of this note.
	
	\begin{theoremM}\label{thmm:1}
		For any rational polyhedral finite fan~$\Sigma$ 
		there is the isomorphism
		\[
		H^{i}(X_{\Sigma};\Q)\cong \bigoplus_{s+t=i} H^{cw}_{n-s}(\Sigma;H^{t}(D_{\Sigma})),\ i\geq 0.
		\]
	\end{theoremM}
	
	Theorem~\ref{thmm:1} (see Theorem~\ref{thm:toricmain} below) is known and was proved by different arguments in~\cite{da-78} (with~$\C$ coefficients) and~\cite{to-14}.
	(In~\cite{fr-06} the orbit spectral sequence collapse was obtained over any commutative unital ring for any nonsingular toric variety.)
	More generally, the spectral sequence induced by the coskeletal filtration for any toric diagram over any CW poset~$\Pp$ abuts at its second page (see Theorem~\ref{thm:orbss}).
	This filtration coincides with the orbit filtration for any \textit{strongly reduced}~$T$-diagram~$D$ (i.e.~$\dim D(p)=\rk T-\rk p$ holds for every~$p\in \Pp$), see~\S\ref{sec:tfac}.
	
	A~$T$-action on~$X$ is called \textit{Cohen-Macaulay}~\cite{go-to-10,fr-pu-11} if the respective Atiyah-Bredon-Franz-Puppe sequence is exact in all terms of degree strictly greater than the least dimension of a~$T$-orbit in~$X$.
	By definition~\cite{go-ko-mc-98}, a~$T$-space is \textit{equivariantly formal} over~$\Q$ if the cohomological Serre spectral sequence (with~$\Q$ coefficients) of the corresponding Borel fibration abuts at its second page.
	If there is at least one fixed point in a~$T$-space~$X$ then the action is Cohen-Macaulay iff $X$ is equivariantly formal (over~$\Q$).
	Given a~$T$-diagram~$D$, the Borel fibration of~$\hoc D$ takes form~$\hoc BS\to BT$, where the diagram~$S\colon \Pp\to\Top$ is given by the stabilizers for the~$T$-action on~$\hoc D$ (see Proposition~\ref{pr:borfib}), and~$B$ is the classifying functor.
	
	\begin{theoremM}\label{thmM:2}
		Let~$D$ be any strongly reduced~$T$-diagram over a CW poset~$\Pp$.
		Then the~$T$-action on~$\hoc D$ is Cohen-Macaulay iff the sheaf~$H^{*}(BS)$ is acyclic.
		In this case, the isomorphisms 
		\[
		H^*_{T}(\hoc D;\Q)\cong \lm H^*(BS),
		\]
		\[
		\Tor^{-i,2t}_{H^{*}(BT)}(H^*_{T}(\hoc D;\Q),\Q)\cong 
		\lm^{t-i} H^{t}(D),
		\]
		of~$H^{*}(BT)$- and~$\Q$-modules, respectively, hold for all~$i,t\geq 0$.
	\end{theoremM}
	
	Theorem~\ref{thmM:2} (see Theorem~\ref{thm:cmdiag} below) follows from MVSS abutment mentioned above.
	Another key tool of the proof is the~$t$-graded \textit{comparison spectral sequence} (CSS, for short) of a~$T$-diagram~$D$ (see Section~\ref{sec:compar}).
	In particular, in the conditions of Theorem~\ref{thmM:2}, equivariant formality is equivalent to odd cohomology vanishing for~$\hoc D$ and existence of fixed points for the~$T$-action (compare with~\cite{ma-pa-06}).
	We mention a recent result on equivariant formality of partial quotients, which is a particular case for homotopy colimits for  toric diagrams~\cite{ku-23}.
	
	For any CW poset~$\Pp$ the subposet of its elements of rank~$\leq q$ is called a~\emph{$q$-skeleton}~$\Pp_{\leq q}$ of~$\Pp$.
	The skeleton of a diagram~$D\colon\Pp\to\Top$ is defined by the restriction of~$D$ to~$\Pp_{\leq q}$.
	In particular, the~$q$-skeleton of~$D_{\Sigma}$ coincides with~$D_{\Sigma_{\leq q}}$, where~$\Sigma_{\leq q}$ is the union of all cones  of dimension~$\leq q$ in~$\Sigma$.
	In~\S\ref{sec:betti} we obtain the description of \textit{bigraded Betti numbers}~$\beta^{-i,2j}(D)$ of~$\hoc D$ (i.e. the dimensions of the second page terms for Eilenberg-Moore spectral sequence of the Borel fibration of~$\hoc D$, or EMSS for short) in terms of~$h$-numbers and reduced Euler characteristic~$\widetilde{\chi}$ as follows.
	
	\begin{theoremM}\label{thmM:3}
		Let~$D$ be a strongly reduced~$T$-diagram over a CW poset~$\Pp$.
		Suppose that~$\hoc D$ is equivariantly formal over~$\Q$.
		Then~$\rk T=\dim \Pp=d$ and one has $\beta^{0,2i}(D)=h_{d-i}(\Pp)$,
		\[
		\beta^{0,2j}(D_{\leq q})=
		\beta^{0,2j}(D),\ 
		\beta^{-i,2(i+q)}(D_{\leq q})=
		(-1)^{q+1}\sum_{p\in \sk^{q} \Pp} \binom{d-\rk p}{i+q}\widetilde{\chi}((\sk^{q} \Pp)_{> p}),\ 
		i\geq 0,\ j<q,
		\]
		and all other bigraded Betti numbers for~$\hoc D_{\leq q}$ vanish.
	\end{theoremM}
	
	As an application of Theorem~\ref{thmM:3}, we describe the ordinary and bigraded Betti numbers of skeletons for toric manifolds and nonsingular compact toric varieties~$X_{\Sigma}$.
	In the case of any skeleton for~$\C P^{m-1}$ we compare the obtained formulas from Theorem~\ref{thmM:3} with those implied by~\cite{fu-19,gr-th-07}.
	
	This note is organized as follows.
	The necessary (for the purposes of this work) notions of (co)sheaf theory on posets are given in Section~\ref{sec:shposet}.	
	In Section~\ref{sec:formality}, MVSS collapse for formal diagrams (based on natural formality over~$\Q$ of tori and their classifying spaces) is recalled after introducing some basic definitions.
	In Section~\ref{sec:tfac} we compare the orbit and coskeletal spectral sequences for strongly reduced~$T$-diagrams and prove Theorem~\ref{thmm:1}.
	In Section~\ref{sec:borel} we recall a diagrammatic description of the Borel space for~$\hoc D$ of a~$T$-diagram~$D$ and the respective EMSS.
	The comparison spectral sequence is introduced in Section~\ref{sec:compar}.
	Definition of equivariant formality, Cohen-Macaulayness for~$T$-spaces and the proof of Theorem~\ref{thmM:2} are given in Section~\ref{sec:eftoric}.
	The concluding Section~\ref{sec:betti} is devoted to computations of bigraded Betti numbers for equivariantly formal spaces and skeletons of toric diagrams, providing the proof of Theorem~\ref{thmM:3}.
	A detailed computation for (bigraded) Betti numbers for skeletons of a complex projective space concludes the note.
	
	{\emph{Acknowledgements.}}
	The author thanks Okayama University of Science (in particular, Shintaro Kuroki) for a comfortable work atmosphere and assistance during his research visit in 2023.
	The author expresses his gratitude to Matthias Franz for spotting an error in integral natural formality argument given in an early version of this note, and to Oliver Goertsches for explaining Cohen-Macaulayness for torus actions without fixed points.
	This work was partially supported by DFG Walter Benjamin Fellowship [Deutsche Forschungsgemeinschaft (DFG) - Projektnummer 561158824], the French national research agency [grant ANR-20-CE40-0016], and Young Russian Mathematics 2023 award.
	
	\section{(Co)homology of (co)sheaves on posets}\label{sec:shposet}
	
	In this section we recall the basic notions of cellular (co)sheaf theory over Alexandrov spaces by following~\cite{ay-17} (see~\cite{cu-14} for the detailed overview).
	Let~$\Mod_{R}$ be the category of finitely generated graded modules over a commutative associative ring~$R$.
	
	\begin{definition}[\cite{bj-84}]
		A poset~$\Pp=\Pp(X)$ is called a \textit{CW poset} if~$\Pp$ is the poset of cells~$\lb \sigma_{p}\rb_{p\in \Pp}$ in a regular CW complex~$X$.
		Denote by~$\hat{0}$ the least element of~$\Pp$ corresponding to the empty cell.
		The CW poset~$\Pp$ admits a unique (i.e. not depending on~$X$) rank function~$\rk\colon \Pp\to \Z$ defined by~$\rk c:=\dim \sigma_c+1 \geq 0$.
		Let~$\dim \Pp:=\dim X$ be the (well-defined) dimension of~$\Pp$.
	\end{definition}
	
	In this note, we consider only finite posets, unless explicitly stated otherwise.
	By abuse of the notation, we denote by~$\Pp$ the category, whose objects are all objects of the poset~$\Pp$ and arrows are the corresponding partial order relations.
	For any two elements~$p, q\in \Pp$ write~$p\leq_{i} q$ if~$p\leq q$ and and~$\rk p=\rk q-i$.
	Given any pair~$p\leq_{1} q$ in~$\Pp=\Pp´(X)$, denote by~$[q:p]:=[\sigma_{q}:\sigma_{p}]$ the corresponding index in the regular CW complex~$X$.
	(The values of such indices are~$-1,0,1$.)
	
	\begin{definition}[\cite{cu-14}]\label{def:cellcoh}
		A covariant (contravariant, respectively) functor~$\mc{A}\colon \Pp\to \Mod_{R}$, is called a \emph{cellular sheaf} (\emph{cellular cosheaf}, respectively) with values in~$\Mod_{R}$.
		The \emph{cellular cohomology}~$H^*_{cw}(\Pp;\mc{A})$ of a cellular sheaf~$\mc{A}$ on~$\Pp$ is by definition the cohomology~$R$-module of the chain complex~$(C^*_{cw}(\Pp,\mc{A}),d)$ (depending on~$\Pp$ and the index numbers of CW complex~$X$ corresponding to~$\Pp$) defined by
		\[
		C^i_{cw}(\Pp,\mc{A}):=
		\bigoplus_{p\in \Pp\colon \rk p=i} \mc{A}(p),\ 
		d(p\times y):=\sum_{p\leq_{1} p'}[p':p]\cdot p'\times \mc{A}(p\leq p')y.
		\]
		(Here, we identify finite products with coproducts in an abelian category.)
		The cellular homology~$H_*^{cw}(\Pp;\mc{B})$ of a cellular cosheaf~$\mc{B}$ on~$\Pp$ is the homology~$R$-module of the chain complex~$(C^{cw}_*(\Pp,\mc{B}),d)$,
		\[
		C_i^{cw}(\Pp,\mc{B}):=\bigoplus_{p\in \Pp\colon \rk p=i} \mc{B}(p),\ 
		d(p\times y):=\sum_{p'\leq_{1} p}[p:p'] \cdot p'\times \mc{B}(p'\leq p)y.
		\]
	\end{definition}
	
	One can see that the cellular (co)homology does not depend on the choice of~$X$ and indices.
	By a slight abuse of the notation we call a cellular (co)sheaf by a (co)sheaf in the sequel.
	
	\begin{remark}
		For any sheaf~$\mc{A}$ on any CW poset~$\Pp$ of dimension~$d$ one has
		\[
		H_{cw}^{s}(\Pp;\mc{A})=
		H^{cw}_{d-s}(\Pp^{op};\mc{A}),
		\]
		by simply considering~$\mc{A}$ as the cosheaf on~$\Pp^{op}$.
	\end{remark}
	
	\begin{example}
		For any diagram~$D\colon\Pp\to\Top$, the functor~$H^*(D(-);\Bbbk)$ is a cosheaf of graded~$\Bbbk$-algebras on~$\Pp$, or equivalently, a sheaf on~$\Pp^{op}$.
	\end{example}
	
	For any elements~$p<_{2} q$ of a CW poset~$\Pp$, the open interval~$(p,q)$ in~$\Pp$ consists precisely of two distinct elements~$\lb p',q'\rb$, and the identity
	\begin{equation}\label{eq:intid}
		[q:p'][p':p]+[q:q'][q':p]=0,
	\end{equation}
	holds~\cite{bj-84}.
	This implies that~$d^2=0$ holds for cellular cosheaf homology and sheaf cohomology.
	
	\begin{definition}
		For any poset~$\Pp$ the \textit{order complex}~$\ord \Pp$ of~$\Pp$ is the simplicial complex of all strictly monotone chains in~$\Pp$.
		Notice that~$\ord \Pp$ is a cellular poset with the index numbers
		\[
		[\sigma:\partial_{i}\sigma]=(-1)^{n-i},\ \sigma=(p_0>p_1 >\cdots >p_{n}),\ \partial_{i}\sigma:=(p_0>\cdots >\widehat{p_{i}}>\cdots >p_{n}),
		\]
		equal to those of the simplicial complex~$\ord \Pp$, where the hat means omitting the corresponding element in the chain.
	\end{definition}
	
	\begin{definition}[\cite{ay-17}]
		For any sheaf~$\mc{A}\colon \Pp\to \Mod_{R}$, the \textit{refinement sheaf}~$s\mathcal{A}\colon \ord \Pp\to \Mod_{R}$ is the sheaf given by the formula
		\[
		s\mathcal{A}(p_0 > p_1 >\cdots> p_{s}):=\mc{A}(p_{0}),
		\]
		\[
		s\mathcal{A}((p_0 > p_1 >\cdots> p_{s})\to (q_0 > q_1 >\cdots> q_{t}) ):=\mc{A}(p_0\to q_{0}),
		\]
		For a morphism of sheaves~$\alpha\colon \mc{A}\to \mc{B}$ (that is, a natural transformation of the corresponding functors) define the induced morphism~$s\alpha\colon s\mc{A}\to s\mc{B}$ of sheaves by the formula
		\[
		s\mc{A}(p_0 > p_1 >\dots > p_{s})=
		\mc{A}(p_{0})\overset{\alpha}{\to}\mc{B}(p_{0})=
		s\mc{B}(p_0 > p_1 >\dots > p_{s}).
		\]
	\end{definition}
	
	\begin{remark}
		The choice of this notation for refinement is explained as follows.
		The sheaf~$s \mc{A}$ may be regarded as the restriction of the simplicial replacement~$\srep_{\bullet}\mc{A}\colon\Delta^{op}\to\Mod_{R}$ of~$\mc{A}$ to the collection of nondegenerate simplices in the nerve~$N_{\bullet}(\Pp)$.
		The refinement cosheaf~$c\mc{B}$ of a~cosheaf~$\mathcal{B}$ on~$\ord \Pp$ is defined in a similar way to the refinement sheaf.
	\end{remark}
	
	\begin{definition}[\cite{cu-14}]\label{def:shdef}
		By definition, for a sheaf~$\mc{A}\colon\Cc\to\Mod_{R}$ over a small category~$\Cc$ the corresponding sheaf cohomology~$H^*(\Cc;\mc{A})$ are the cohomology~$R$-modules of the cochain complex~$C_*^{cw}(\ord \Cc;s\mc{A})$.
		The sheaf~$\mc{A}\colon \Cc\to \Mod_{R}$ is called \textit{acyclic} if~$H^{i}(\Cc;\mc{A})$ vanishes unless~$i=0$.
	\end{definition}
	
	The proof of the following proposition is well known (namely, it follows directly by the universal~$\delta$-functor theorem~\cite{we-97}).
	
	\begin{proposition}[\cite{cu-14}]\label{prop:limcoh}
		For any cellular sheaf~$\mc{A}$ on a CW poset~$\Pp$ with values in~$\Mod_{R}$, one has the isomorphism of~$R$-modules
		\[
		\lmm^{*} \mc{A} \cong 
		H^*(\Pp;\mc{A}).
		\]
	\end{proposition}
	
	In the sequel, we consider cellular cosheaves, their cellular homology and respective sheaf cohomology.
	Namely, given a diagram~$D$ of topological spaces over a CW poset~$\Pp$, one has the cellular cosheaf~$H^*(D)$ on~$\Pp$, its cellular homology~$H_{*}^{cw}(\Pp;H^*(D))$ and sheaf cohomology 
	\[
	H^{*}(\Pp^{op};H^*(D))\cong \lmm^{*} H^*(D).
	\]
	In the expression~$\lm^{*} H^*(D)$ we consider the covariant functor~$H^*(D)$ defined on the category~$\Pp^{op}$.
	
	For the proof of the following theorem see~\cite[Thm.~2; \S~4.1]{ev-tu-15}.
	
	\begin{theorem}[{\cite{ev-tu-15}}]\label{thm:evtu}
		For any sheaf~$\mc{A}$ (cosheaf~$\mc{B}$, respectively) of~$R$-modules with free (as~$R$-modules) stalks over any CW poset~$\Pp$ one has isomorphisms of~$R$-modules
		\[
		H^{*}(\Pp;\mc{A})\cong
		H^{*}_{cw}(\Pp;\mc{A}),\ 
		H_{*}(\Pp;\mc{B})\cong
		H_{*}^{cw}(\Pp;\mc{B}).
		\]
	\end{theorem}
		
	\section{Spectral sequence of a diagram and its abutment}\label{sec:formality}
	
	The purpose of the current section is to recall the spectral sequence abutment for any toric diagram (Theorem~\ref{thm:rcollf}) and its classifying space (Theorem~\ref{thm:btoric}) for cohomology with~$\Q$ coefficients established in~\cite{so-to-26}.
	We refer to~\cite{go-ja-09} for the basic definitions from simplicial homotopy theory.

    \begin{definition}
		Let~$\Delta$,~$\Delta_{n}$ be the categories of all finite ordered sets, and of all ordered sets of cardinality~$\leq n$, respectively.
		Denote by~$\Top$ the category of compactly generated weakly Hausdorff topological spaces and continuous maps between them.
		For a simplicial space~$X_{\bullet}\in \Top^{\Delta^{op}}=\stp$ let
		\[
		\tr_{n} X_{\bullet}:=\iota^{n}_{*} X_{\bullet}\in \Top^{\Delta_{n}^{op}},
		\]
		where~$\iota^{n}\colon \Delta_{n}\to \Delta$ is the natural inclusion functor.
		Let
		\[
		\sk_{n}\colon \Top^{\Delta_{n}^{op}}\to \Top^{\Delta^{op}},
		\]
		be the left adjoint of~$\tr_{n}$ given by the respective left Kan extension.
		The functor
		\[
		\Sk_{n}:=\sk_{n}\circ\tr_{n}\colon \stp\to \stp,
		\]
		is called the \textit{$n$-skeleton} of~$X_{\bullet}$.
	\end{definition}
	
	Let~$\widetilde{X}_{\bullet}\to X_{\bullet}$ be a cofibrant replacement of a simplicial space.
	The counit of the adjunction~$sk_{n}\dashv \tr_{n}$ is the map~$\Sk_{n} \widetilde{X}_{\bullet}\to \widetilde{X}_{\bullet}$ which is a cofibration.
	The simplicial space~$X_{\bullet}$ has dimension~$n$ if~$X_{\bullet}=\Sk^{n} X_{\bullet}$.
	In what follows, we consider only simplicial spaces of finite dimension, whose every object is cofibrant.
	Using the Quillen model structure, we often replace a weak equivalence in~$\Top$ with a homotopy equivalence by Whitehead theorem for diagrams consisting of CW complexes.
	By taking the geometric realization of an~$n$-dimensional simplicial space, one obtains the filtration
	\begin{equation}\label{eq:filt}
		\lv\Sk_{0} \widetilde{X}_{\bullet}\rv\subset \cdots \subset \lv\Sk_{n} \widetilde{X}_{\bullet}\rv=\lv\widetilde{X}_{\bullet}\rv,
	\end{equation}
	where~$\lv X_{\bullet}\rv$ denotes the realization of a simplicial space.
	(Notice that~$\lv\widetilde{X}_{\bullet}\rv\to \lv X_{\bullet}\rv$ is a weak equivalence.)
	
	\begin{definition}
		For a diagram~$D\colon\Cc\to\Top$ over a small category~$\Cc$ the \textit{simplicial replacement}~$\srep_{\bullet} D\in \stp$ of~$D$ is given by the set of~$n$-simplices
		\[
		\biggl\lb\bigsqcup_{\sigma\in N_{n}(\Cc)} D(c_{n})|\ \sigma=(c_0\ot \cdots\ot c_{n})\in N_{n}(\Cc)\biggr\rb,
		\]
		where~$N_{\bullet}(\Cc)\in \sset$ is the nerve category of~$\Cc$.
		The face and degeneracy maps are given by the standard formula, which can be found e.g. in~\cite{go-ja-09}.
	\end{definition}
	
	\begin{definition}[\cite{bo-ka-72, bo-87}]\label{def:bkss}
		The cohomological spectral sequence of a simplicial space~$X_{\bullet}\in \stp$ is the first quadrant spectral sequence~$((E_{X_{\bullet}})_{r}^{p,q},d_r)$,~$r\geq 1$, in cohomology with coefficients in a PID~$\Bbbk$, corresponding to the filtration \eqref{eq:filt}.
		The spectral sequence~$((E_{D})_{r}^{p,q},d_r)$ (in singular cohomology~$H^{*}(-;\Bbbk)$) of a diagram~$D\colon\Cc\to\Top$ over~$\Cc$ with values in~$\Top$ is the respective spectral sequence of the simplicial replacement~$\srep_{\bullet} D$~\cite{go-ja-09}.
	\end{definition}
	
	\begin{definition}
		The realization~$\lv\srep_{\bullet} D\rv$ of the simplicial replacement for a~$\Cc$-diagram~$D$ of spaces is called the \textit{homotopy colimit}~$\hoc D$ of~$D$.
	\end{definition}
		
	\begin{remark}\label{rem:bkss}
		The simplicial space~$\srep D_{\bullet}$ is cofibrant provided that~$D(c)$ is a cofibrant space for every~$c\in\Cc$~\cite{go-ja-09}.
		It is well known that the spectral sequence second page groups for~$D$ are right derived limits of the respective presheaf~$H^{*}(D)$ (e.g. see~\cite{no-ra-05}).
		Therefore, the spectral sequence for a diagram~$D\colon \Cc\to\Top$ takes the following form:
		\[
		(E_{D})_{2}^{i,j}={\lm}^{i} H^j(D)\Rightarrow H^{i+j}(\hoc D).
		\]
	\end{remark}
	
	\begin{definition}[\cite{we-zi-zi-99}]
		A diagram~$D\colon \Cc\to\Top$ consisting of compact tori~$D(c)=(S^{1})^{d(c)}$ (as objects) and group homomorphisms (as arrows) is called a \emph{toric diagram}, where~$d(c)\in\Z_{\geq 0}$,~$c\in \Cc$.
		More generally, a \textit{quasitoric diagram} has quasitori (i.e. product of a finite abelian group with a compact torus) as objects, and group homomorphisms as its arrows.
	\end{definition}
	
	\begin{theorem}[{\cite[Theorem~2]{so-to-26}}]\label{thm:rcollf}
		For any toric diagram~$D$ the cohomological spectral sequence $((E_{D})_{*}^{*,*},d_{*})$ over~$\Q$ abuts at the second page.
		One has
		\[
		H^{n}(\hoc D;\Q)\cong \bigoplus_{i+j=n} \lmm^{i} H^{j}(D;\Q),\ n\geq 0.
		\]
	\end{theorem}
	
	Let~$\alpha\colon D\to \kappa T$ be a morphism of quasitoric diagrams, where~$\kappa T$ denotes the constant diagram.
	
	\begin{theorem}[{\cite[Theorem~2]{so-to-26}}]\label{thm:btoric}
		For any toric diagram~$D$ the cohomological spectral sequence $((E_{BD})_{*}^{*,*},d_{*})$ of~$H^*(BT;\Q)$-modules abuts at the second page.
		One has a $\Q$-vector space isomorphism
		\[
		H^{n}(\hoc BD;\Q)\cong \bigoplus_{i+j=n} \lmm^{i} H^{j}(BD;\Q),\ n\geq 0.
		\]
	\end{theorem}
	
	From now on we assume that~$\Cc=\Pp$ is a finite graded poset.
	The above spectral sequence for a diagram~$D\colon \Pp\to\Top$ is isomorphic to Mayer-Vietoris spectral sequence of the open cover induced by open stars in~$\Pp$.
	In this case, we call the spectral sequence from Definition~\ref{def:bkss} the Mayer-Vietoris spectral sequence (MVSS, for short), e.g. see~\cite[\S5]{so-to-26}.
	
	\section{Torus actions arising from~$T$-diagrams}\label{sec:tfac}
	
	In this section, we recall the definition of a~$T$-diagram and some related basic notions for the respective~$T$-space~$\hoc D$.
	We derive the orbit spectral sequence collapse over~$\Q$ (Theorem~\ref{thm:orbss}) which is well known for toric varieties~\cite{da-78},~\cite{to-14} by a different argument (Theorem~\ref{thm:toricmain}).
	
	\subsection{Faces and orbits}\label{ssec:face}
	
	We recall the definitions of the orbit filtration and faces for a~$T$-space below in the following setting.
	Let~$X$ be a finite regular~$T$-CW complex.
	Then~$X$ is a Hausdorff compact~$T$-space.
	This implies that the orbit space~$Q=X/T$ is a regular CW complex, and that there are only finitely many distinct orbit types in~$X$.
	
	\begin{definition}
		For an effective~$T$-action on a topological space~$X$ consider the \textit{orbit type filtration}~\cite{br-72}
		\begin{equation}\label{eq:otf}
			\varnothing=X_{-1}\subset X_{0}\subset \cdots\subset X_{k}=X,
		\end{equation}
		where~$k$ is the rank of~$T=(S^1)^k$.
		It induces the filtration of the orbit space~$Q=X/T$
		\[
		\varnothing=Q_{-1}\subset Q_{0}\subset \cdots\subset Q_{k}=Q,
		\]
		by putting~$Q_{i}:=\pi(X_{i})$, where~$\pi\colon X\to Q$ is the quotient map.
		Any connected component of~$Q_{i}\setminus Q_{i-1}$ is called an \textit{open face} of~$X$.
		Assume that the indexing set~$\Pp=\Pp(X)$ for the collection~$\lb Q^{\circ}_{p}\rb_{p\in \Pp}$ of all open faces in~$X$ is finite.
		The partial order on~$\Pp^{op}$ is defined by~$p\leq q$ iff~$Q_{p}\subseteq Q_{q}$, where~$Q_{p}:=\overline{Q^{\circ}_{p}}$ and the bar denotes the closure in~$Q$.
		The set~$Q_{p}$ is called a \textit{closed face} of the~$T$-space~$X$.
		The poset~$\Pp^{op}$ is called the \textit{orbit poset} of the~$T$-space~$X$.
		The (co)homological spectral sequence of the filtration~\ref{eq:otf} is called the \textit{orbit spectral sequence} of the~$T$-action on~$X$.
	\end{definition}
	
	\begin{remark}
		Notice that~$Q$ is a face of~$X$ iff the respective orbit poset~$\Pp^{op}$ has the greatest element.
		The orbit poset of any toric variety of complex dimension~$n$ has the greatest element corresponding to zero of the fan (because there is a dense~$(\C^{\times})^n$-orbit, and the corresponding fan has the least element).
		For any smooth compact~$T$-manifold such a condition is also satisfied by the principal orbit theorem~\cite[Theorem 3.1]{br-72}.
	\end{remark}
	
	\begin{definition}\label{def:tdiag}
		We call a morphism of diagrams~$\alpha\colon \kappa T\to D$ (i.e. a natural transformation of functors) over a finite poset~$\Pp$ a \textit{$T$-diagram}, if any arrow in~$\alpha$ is an epimorphism of groups.
	\end{definition}
	
	Any~$T$-diagram~$D\colon\Pp\to\Top$ is a diagram with values in the category of~$T$-spaces.
	It induces an effective~$T$-action on the topological space~$\hoc D$.
	The~$T$-action is transitive on each object of~$D$.
	Therefore, the stabilizers of the~$T$-action on~$D(p)$ depend only on~$p\in \Pp$, which we denote by~$S_{p}\subseteq T$.
	The collection~$\lb S_{p}|\ p\in \Pp\rb$ with the natural embeddings of the isotropy subgroups gives a well-defined quasitoric diagram~$S=S(D)$ over~$\Pp$.
    (With the above definition of $\Pp$, $S_{p}\subseteq S_{q}$ for $p\leq q$, i.e. $Q_{p}\supseteq Q_{q}$.)
    
	\begin{definition}
		We call~$S\colon \Pp\to \Top$ the \textit{stabilizer diagram} of a~$T$-diagram~$D$.
	\end{definition}
	
	\begin{remark}
		It follows directly from the definition of a~$T$-diagram~$D$ that~$S_{p}\subseteq S_{q}$ holds for any~$p\leq q\in \Pp$, where~$S$ is the stabilizer diagram of~$D$.
		In particular, a closure of any orbit stratum in~$\hoc D$ is a union of orbit types of non-increasing dimension.
	\end{remark}
	
	By the orbit-stabilizer correspondence, there is a short objectwise-exact sequence of diagrams over~$\Pp$ with values in groups
	\begin{equation}\label{eq:stabses}
		\begin{tikzcd}
			\kappa 1\arrow{r} & S \arrow{r} & \kappa T \arrow{r} & D \arrow{r} & \kappa 1.
		\end{tikzcd}
	\end{equation}
	
	\begin{definition}\label{def:red}
		We call a~$T$-diagram~$D$ over~$\Pp$ a \textit{reduced diagram} if the inequality~$S_p\neq S_q$ holds for any~$p<q\in \Pp$.
		We call a~$T$-diagram~$D$ a \textit{strongly reduced diagram} if one has~$\dim S_p= \rk p$ for any~$p\in \Pp$.
	\end{definition}
	
	\begin{remark}
		Any strongly reduced~$T$-diagram~$D$ satisfies~$\dim D(p)=\rk T-\rk p$ for any~$p\in \Pp$.
		In particular,~$\dim \Pp=\dim T-b$, where~$b$ is the lowest dimension of a~$T$-orbit in~$\hoc D$.
	\end{remark}
	
	\begin{definition}[\cite{ay-17}]
		Let~$D\colon\Pp\to\Top$ be any diagram over a CW poset~$\Pp$ of dimension~$d$.
		Define the \textit{coskeletal filtration} on~$\Pp$
		\[
		\varnothing=\Pp_{\geq d+1}\subset \Pp_{\geq d}\subset \cdots \subset \Pp_{\geq 0}=\Pp,\ \Pp_{\geq i}:=\lb p\in \Pp|\ \rk p\geq i\rb,
		\]
		where~$\Pp_{\geq i}$ is the \textit{$i$-coskeleton} subposet in~$\Pp$.
		It induces the \textit{coskeletal filtration} on~$\hoc D$
		\begin{equation}\label{eq:coskhoc}
			\varnothing=\hoc \tau^{*}_{d+1}D \subset \hoc \tau^{*}_{d}D\subset \cdots \subset\hoc D,
		\end{equation}
		where~$\tau_{i}\colon \Pp_{\geq i}\to \Pp$ is the natural embedding functor.
	\end{definition}
	
	The proof of the following proposition is straightforward.
	
	\begin{proposition}\label{pr:comp}
		Let~$D$ be a~$T$-diagram over a CW poset~$\Pp$.
		Then~$\Pp^{op}$ is isomorphic to the orbit poset of~$\hoc D$ iff~$D$ is reduced.
		If~$D$ is strongly reduced then the coskeletal and orbit filtrations on~$\hoc D$ coincide.
	\end{proposition}
	
	\subsection{Orbit spectral sequence abutment}
	Let~$X=\hoc D$ for a~$T$-diagram~$D$ over a CW poset~$\Pp$.
	Consider the spectral sequence~$((C_{D})_{*,*}^{*},d^{*})$ of the filtration \eqref{eq:coskhoc} for singular cohomology~$H^{*}(-;\Q)$ with rational coefficients.
	The following lemma is proved in a similar way to the proof of~\cite[Theorem 2.5.3]{jo-98} (notice that instead of excision for compactly supported cohomology one can identify the quotients of the respective skeletons with the wedge of iterated suspensions).
	
	\begin{lemma}[\cite{jo-98}]\label{lm:orbf}
		For any toric diagram~$D\colon \Pp\to\Top$ over any CW poset $\Pp$ of~$\dim \Pp=d$ the spectral sequence of the coskeletal filtration on~$\hoc D$ is given as follows:
		\[
		C^{2}_{s,t}=H^{cw}_{d-s}(\Pp;H^{t}(D)) \Rightarrow H^{s+t}(\hoc D;\Q),
		\]
	\end{lemma}
	
	\begin{theorem}\label{thm:orbss}
		For any toric diagram~$D\colon \Pp\to\Top$ over any CW poset $\Pp$ of~$\dim \Pp=d$ the spectral sequence~$((C_{D})_{*,*}^{*},d^{*})$ (over~$\Q$) of the coskeletal filtration abuts at its page $2$: 
		\[
		H^{k}(\hoc D;\Q)\cong \bigoplus_{s+t=k} H^{cw}_{d-s}(\Pp;H^{t}(D)).
		\]
	\end{theorem}
	\begin{proof}
		By Theorem~\ref{thm:rcollf} and Proposition~\ref{prop:limcoh},~$H^{k}(\hoc D)$ is a direct sum of the corresponding cohomology groups of the sheaf~$H^{*}(D)$ on~$\Pp^{op}$.
		By Theorem~\ref{thm:evtu}, these are isomorphic to the cellular homology of~$H^{*}(D)$.
		The proof is complete.
	\end{proof}
	
	We conclude this section by deducing the orbit spectral sequence collapse for any toric variety.
	Let~$X_{\Sigma}$ be any toric variety of complex dimension $n$ with the corresponding rational polyhedral fan~$\Sigma$ in~$N_{\Q}=N\otimes\Q$,~$N=\Hm(S^1,T^n)$.
	(Here it is required that any cone~$\sigma\in\Sigma$ is \textit{salient}, i.e. contains no lines from~$N_{\Q}$; we also assume that~$\Sigma$ consists of only finitely many cones.)
	The variety~$X_{\Sigma}$ is equipped with the complex-analytical topology.
	Consider the diagram~$D_{\Sigma}\colon\Sigma\to\Top$, where~$D_{\Sigma}(\sigma)$ is the compact form of the algebraic torus corresponding to the orbit~$X_{\sigma}$.
	The arrow~$D_{\Sigma}(\sigma\leq \tau)$ is defined by taking quotient by a larger isotropy subgroup,~$\sigma\subseteq\tau\in\Sigma$.
	The~$T$-equivariant homeomorphism (homotopy equivalence, respectively) of the underlying topological space for compact (any, respectively)~$X_{\Sigma}$ with~$\hoc D_{\Sigma}$ is proved in~\cite{fr-10}.
	
	By a slight abuse of the notation we denote the poset of cones in~$\Sigma$ ordered by inclusion relation by~$\Sigma$.
	This poset is the face poset of the regular CW complex of cones in the polyhedron~$|\Sigma|$.
	Since~$\Sigma_{<\sigma}$ is a sphere for any~$\sigma\in\Sigma$,~$\Sigma$ is a CW poset.
	One can deduce from the salient property of the cones that~$D_{\Sigma}$ is a strongly reduced~$T$-diagram.
	The next theorem is a direct corollary from Theorem~\ref{thm:orbss} and Proposition~\ref{pr:comp}.
	
	\begin{theorem}[\cite{da-78,to-14}]\label{thm:toricmain}
		For any rational polyhedral finite fan~$\Sigma$ of dimension $n$ one has
		\[
		H^{i}(X_{\Sigma};\Q)\cong \bigoplus_{s+t=i} H^{cw}_{n-s}(\Sigma;H^{t}(D_{\Sigma})).
		\]
	\end{theorem}
	
	\begin{remark}
		It was conjectured in~\cite{fr-06} that the orbit spectral sequence for any toric variety collapses on its second page with trivial additive abutment for cohomology in~$\Z$.
		(For the proofs of collapse over~$\C$ and~$\Q$ see~\cite{da-78,to-14}, respectively.)
		The proof of Theorem~\ref{thm:orbss} in fact implies for the compact toric variety with any rational polyhedral fan that trivial additive abutments at the second pages for the respective orbit and Mayer-Vietoris spectral sequences (with any coefficients) are equivalent.
        Notice that the homotopy equivalence~\cite{fr-10} $X_{\Sigma}\simeq \hoc D_{\Sigma}$ is not proper in general, and the compactly supported cohomology is not invariant with respect to it.
        Therefore, the compactly supported cohomology for~$X_{\Sigma}$ do not reduce to the cohomology of~$\hoc D_{\Sigma}$ if~$X_{\Sigma}$ is not compact.
	\end{remark}
		
	\section{Borel space of a homotopy colimit and EMSS}\label{sec:borel}
	
	In this section we describe the homotopy decomposition for the Borel fibration of~$\hoc D$ for any~$T$-diagram~$D$.
	Then we recall the definition of the Eilenberg-Moore spectral sequence in the case of the Borel fibration.
	These facts are well known (for instance, see~\cite[\S 4.1]{df-04}).
	
	Let~$\Pp$ be any connected (i.e.~$\lv\ord \Pp\rv$ is connected) finite poset.
	By \eqref{eq:stabses}, for any~$T$-diagram~$D\colon\Pp\to\Top$ there is the respective stabilizer diagram~$S$.
	For any $T$-space $X$ denote by $B_{T} X:=X\times_{T} ET$ the respective Borel space.
	By naturality of the Borel construction, for any~$T$-diagram~$D$ there is the well-defined diagram~$B_{T} D$ given by objectwise application of the functor~$B_{T}$ to~$D$.
	
	\begin{lemma}\label{lm:borelsp}
		For any~$T$-diagram~$D$ over any connected finite poset~$\Pp$, there is a commutative diagram
		\begin{equation}\label{eq:borediag}
			\begin{tikzcd}
				B_{T} \hoc D \arrow{dr} \arrow{r}{\sim} & \hoc B_{T} D\arrow{d}\arrow{r}{\sim} & \hoc BS \arrow{dl}\\
				& BT &
			\end{tikzcd}
		\end{equation}
		where the horizontal arrows are homotopy equivalences.
	\end{lemma}
	\begin{proof}
		Consider~$D\to 1$ as a morphism of diagrams over the category~$\Pp\times T$.
		There is a commutative diagram
		\begin{equation}\label{eq:bordiag}
			\begin{tikzcd}[sep=.3cm]
				\hoc B_{T} D \arrow[equal]{r} & \hoc_{\Pp}\hoc_{T} D \arrow{d} \arrow{r}  & \hoc_{\Pp}\hoc_{T} \kappa 1 \arrow{d}\\
				B_{T} \hoc D\arrow[equal]{r} & \hoc_{T}\hoc_{\Pp} D \arrow{r} & \hoc_{T}\hoc_{\Pp} \kappa 1
			\end{tikzcd}
		\end{equation}
		By Fubini theorem for homotopy colimits~\cite[Theorem 24.9]{ch-sc-02} (or by applying commutation of colimits to a respective cofibrant replacement), both vertical arrows in \eqref{eq:bordiag} are homotopy equivalences.
		The right vertical arrow in \eqref{eq:bordiag} is the identity of~$BT\times |\ord \Pp|$.
		The post-composition of these arrows with the projection gives the Borel fibration over~$BT$.
		This proves the claims about the left triangle in \eqref{eq:borediag}.
		The map~$B_{T} D\to BS$ is a homotopy equivalence of diagrams, because the~$T$-action is transitive on objects of~$D$ with kernel~$S$.
		This implies the claim about the right triangle in \eqref{eq:borediag}.
		The proof is complete.
	\end{proof}
	
	Lemma~\ref{lm:borelsp} directly implies the following (see~\cite{li-so-22}).
	
	\begin{proposition}[\cite{li-so-22}]\label{pr:borfib}
		Let~$D$ be any~$T$-diagram over any connected finite poset~$\Pp$.
		Then the Borel construction for the~$T$-space~$\hoc D$ takes form
		\begin{equation}\label{eq:borfib}
			\hoc D\to \hoc BS\to BT.
		\end{equation}
	\end{proposition}
	
	For a Serre fibration~$p\colon E\to B$ with a connected fiber~$F$ the Eilenberg-Moore spectral sequence~$(E^{*,*}_{*},d)$ of the fiber inclusion is~\cite[p.233]{mcc-01}
	\[
	E^{-i,j}_{2}=\Tor^{-i,j}_{H^*(B)}(H^*(E),\Q)\Rightarrow H^{j-i}(F;\Q),
	\]
	where the first grading is cohomological (i.e. corresponds to element of the resolution of~$H^*(B)$-module~$H^*(E)$) and the second is inner (i.e. coincides with the degree of~$H^*(E)$).
	(We follow the negative grading notation that was used in~\cite{bu-pa-15}, for example.)
	If~$B$ is simply-connected, then~$(E^{*,*}_{*},d)$ converges strongly to~$H^*(F)$, see~\cite[p.233]{mcc-01}.
	
	By applying Eilenberg-Moore spectral sequence to the fibration \eqref{eq:borfib} one obtains the following.
	
	\begin{corollary}\label{cor:emss}
		For any~$T$-diagram~$D$ over any finite poset~$\Pp$, the $(-i,t)$-term of the second page for the corresponding Eilenberg-Moore spectral sequence for the Borel fibration of the~$T$-space~$\hoc D$ is isomorphic to
		\[
		\Tor^{-i,t}_{H^{*}(BT)}(H^{*}(\hoc BS),\Q)\Rightarrow H^{t-i}(\hoc D;\Q).
		\]
	\end{corollary}
	
	\section{EMSS and MVSS}\label{sec:compar}
	
	In this section, we introduce a comparison spectral sequence (CSS, for short), relating the Eilenberg-Moore and Mayer-Vietoris spectral sequences for any~$T$-diagram~$D$ (see Theorem~\ref{thm:bkem}).
	If the sheaf~$H^*(BS)$ is acyclic, then CSS abuts at the second page and gives isomorphism of EMSS and MVSS second pages, up to a change of the corresponding indices.
	
	Let~$D$ be any~$T$-diagram over any finite poset~$\Pp$.
	We need the following auxiliary lemma.
	
	\begin{lemma}\label{lm:torstalk}
		Let~$T\to T'$ be a monomorphism of tori. Then
		\[
		\Tor^{-p,2q}_{H^*(BT')}(H^{*}(BT),\Q)\cong
		\begin{cases}
			H^{q}(T'/T),\ p=q;\\
			0,\ \mbox{ otherwise}.
		\end{cases}
		\]
	\end{lemma}
	\begin{proof}
		Consider the Koszul resolution~\cite{we-97} 
		\[
		\begin{tikzcd}[sep=.3cm]
			0\arrow{r} & H^q(T'/T)\arrow{r} & H^{q-1}(T'/T)\otimes H^{2}(BT') \arrow{r} & \cdots \arrow{r} & H^{2q}(BT')\arrow{r} & H^{2q}(BT)\arrow{r} & 0,
		\end{tikzcd}
		\]
		defined by the short exact sequence
		\[
		\begin{tikzcd}
			0\arrow{r} & H^{1}(T'/T)\arrow{r} & H^{2}(BT')\arrow{r} & H^{2}(BT)\arrow{r} & 0.
		\end{tikzcd}
		\]
		Applying the functor~$-\otimes_{H^*(BT')}\Q$ one obtains trivial cohomology except the boundary term in the cohomological degree~$-q$.
		This gives the required formula by the definition of~$\Tor$ as a left derived functor of tensor product~\cite{we-97}.
	\end{proof}
	
	\begin{definition}
		Denote the~$(-i)$-th group of the Koszul resolution
		\begin{equation}\label{eq:koszulco}
			\begin{tikzcd}[sep=.5cm]
				0 \arrow{r} & \Lambda^{t}(T)\otimes H^{*}(BT) \arrow{r}{f_{t,1}} & \Lambda^{t-1}(T)\otimes H^{*}(BT) \arrow{r}{f_{t,t}} & \cdots \arrow{r} & H^{*}(BT) \arrow{r} & \Bbbk \arrow{r} & 0,
			\end{tikzcd}
		\end{equation}
		of the trivial~$H^{*}(BT)$-module~$\Q$ by~$R^{i}$.
		Let~$D$ be a~$T$-diagram, with the associated stabilizer diagram~$S$ (see \eqref{eq:stabses}).
		Define the~$t$-graded bicomplex~$(C^{i,j}_{t}(D),d_{h},d_{v})$, where (see Definition~\ref{def:cellcoh})
		\[
		C^{i,j}_{t}(D):=C^{j}(\Pp^{op};(R^{2t-i}\otimes H^{*}(BS))^{2t}),
		\]
		power~$2t$ denotes taking the respective grading of the~$H^*(BT)$-module,~$d_{h}$ is the differential for the sheaf cohomology (increasing~$j$ by~$1$) and~$d_{v}$ is induced by the morphism from the Koszul resolution (increasing~$i$ by~$1$).
	\end{definition}
	
	\begin{lemma}\label{lm:kunneth}
		For any cosheaf~$\mathcal{F}$ on~$\Pp$ and any~$\Q$-module~$V$ one has
		\[
		H^{i}(\Pp^{op};\mathcal{F}\otimes \kappa V)\cong H^{i}(\Pp^{op};\mathcal{F})\otimes V.
		\]
	\end{lemma}
	\begin{proof}
		Follows directly from the definitions by flatness of the~$\Q$-module~$V$.
	\end{proof}
	
	\begin{lemma}\label{lm:bicomp}
		One has
		\[
		H^{i}_{v}(H^{j}_{h}(C^{*,*}_{t}(D)))=
		\Tor^{-(2t-i),2t}_{H^{*}(BT)}(\lm^j H^{*}(BS),\Q);
		\]
		\[
		H^{j}_{h}(H^{i}_{v}(C^{*,*}_{t}(D)))=
		\begin{cases}
			\lm^j H^{t}(D),\ i=t,\\
			0,\ i\neq t.
		\end{cases}
		\]
	\end{lemma}
	\begin{proof}
		We make the following computation.
		\begin{multline}\label{eq:hv}
			H^{j}_{h}(C^{i,*}_{t}(D))=
			H^{j}(\Pp^{op};(R^{2t-i}\otimes H^{*}(BS))^{2t} )\cong
			\biggl(R^{2t-i}\otimes H^{j}(\Pp^{op};H^{*}(BS))\biggr)^{2t}\cong\\
			\cong\biggl(R^{2t-i}\otimes \lmm^{j} H^{*}(BS)\biggr)^{2t}.
		\end{multline}
		The last isomorphism is given by Lemma~\ref{lm:kunneth}.
		The above isomorphisms respect the vertical differential whose action on the right-hand side of \eqref{eq:hv} is induced by~$f_{t,*}\otimes \Id$, see \eqref{eq:koszulco}.
		Therefore, the first equality of the lemma follows by the definition of~$\Tor$-groups using the resolution~$(P^{*},f_{t,*})$.
		By the definition of the~$\Tor$-groups (using the resolution \eqref{eq:koszulco}), we find
		\[
		H^{j}_{h}(H^{i}_{v}(C^{*,*}_{t}(D)))=
		H^{j}(\Pp^{op};\Tor^{-(2t-i),2t}_{H^{*}(BT)}(H^{*}(BS),\Q)).
		\]
		The second equality of the lemma now follows from this computation by Lemma~\ref{lm:torstalk}.
	\end{proof}
	
	\begin{theorem}\label{thm:bkem}
		For any~$T$-diagram~$D$ there exists the~$t$-graded spectral sequence ($t\in\Z$)
		\begin{equation}\label{eq:newss}
			\Tor^{-(2t-i),2t}_{H^{*}(BT)}(\lm^{j} H^{*}(BS),\Q)\Rightarrow \lm^{i+j-t} H^{t}(D).
		\end{equation}
	\end{theorem}
	\begin{proof}
		Follows directly by considering the bicomplex spectral sequences corresponding to\\ $(C^{i,j}_{t}(D),d_{h},d_{v})$ and using Lemma~\ref{lm:bicomp}.
	\end{proof}
	
	\begin{corollary}\label{cor:acyc}
		Suppose that the sheaf~$H^{*}(BS)$ is acyclic.
		Then
		\begin{equation}\label{eq:collnewss}
			\Tor^{-i,2t}_{H^{*}(BT)}(H^{*}(\hoc BS),\Q)\cong \lm^{-i+t} H^{t}(D),
		\end{equation}
		and MVSS and EMSS for~$D$ are isomorphic from the second pages up to a linear change of indices.
	\end{corollary}
	\begin{proof}
		By Theorem~\ref{thm:btoric}, the MVSS for the diagram~$BS$ collapses.
		By the acyclicity assumption this implies the isomorphism
		\begin{equation}\label{eq:bscoll}
			H^{*}(\hoc BS)\cong \lm H^{*}(BS),
		\end{equation}
		of~$H^*(BT)$-modules.
		The spectral sequence \eqref{eq:newss} collapses, which again follows directly by the acyclicity assumption.
		This implies \eqref{eq:collnewss}.
	\end{proof}
	
	\begin{remark}\label{rem:acycnull}
		For an acyclic sheaf~$H^*(BS)$, one has
		\[
		H^i(\Pp^{op};\kappa\Q)\cong H^i(\lv\ord \Pp\rv;\Q)=0,\ i\neq 0,
		\]
		because~$H^0(BS)=\kappa\Q$ holds.
	\end{remark}
	
	\begin{remark}
		The comparison spectral sequence \eqref{eq:newss} measures commutativity of~$\Tor$ and~$\lm$, because by Lemma~\ref{lm:torstalk} one has
		\[
		\lm^{j} \Tor^{-i,2t}_{H^{*}(BT)}(H^{*}(BS),\Q)=
		\begin{cases}
			\lm^{j} H^{t}(D),\ i=t;\\
			0,\ \mbox{ otherwise}.
		\end{cases}
		\]
		The functors~$-\otimes_{H^{*}(BT)} \Q$ and~$\lm$ are right- and left-exact, respectively.
		The result of this section may be regarded as a Grothendieck-type spectral sequence~\cite[pp.150--153]{we-97}.
	\end{remark}
	
	\section{Cohen-Macaulay torus actions on~$T$-diagrams}\label{sec:eftoric}
	
	In this section we prove a new criterion of Cohen-Macaulayness for the homotopy colimit of a~$T$-diagram~$D$ over any CW poset~$\Pp$.
	This notion is a generalization of equivariant formality property for torus actions without fixed points~\cite{go-to-10}.
	We compute the bigraded Betti numbers under Cohen-Macaulayness assumption.
	
	\begin{definition}[\cite{fr-pu-06},~\cite{go-to-10}]
		A~$T$-space~$X$, where~$T$ has rank~$k$ and~$b\geq 0$ is the lowest dimension of a~$T$-orbit in~$X$, is called \textit{Cohen-Macaulay} if the following augmented cochain complex
		\begin{equation}\label{eq:abfp}
			\begin{tikzcd}[column sep=0.5cm]
				0 \arrow{r} & H^*_T(X) \arrow{r} & H^*_T(X_b) \arrow{r} & H^{*+1}_T(X_{b+1},X_{b}) \arrow{r} & \cdots \arrow{r} & H^{*+k-b}_{T}(X_{k},X_{k-1}) \arrow{r} & 0,
			\end{tikzcd}
		\end{equation}
		of~$H^{*}(BT)$-modules over~$\Q$ is exact.
		The grading is given by~$-1,0,\dots,k-b$; i.e. zero cohomology is the augmentation term~$H^*_T(X)$.
		If~$b=0$ then the Cohen-Macaulay~$T$-action is called \textit{equivariantly formal} over~$\Q$.
	\end{definition}
	
	The proof of the following fact is similar to the proof of~\cite[Proposition 2.4]{ay-ma-23} (compare with Lemma~\ref{lm:orbf}; also see~\cite[Theorem 2.5.3]{jo-98}).
	
	\begin{lemma}\label{lm:compspec}
		For any strongly reduced~$T$-diagram~$D$ over any CW poset~$\Pp$ the~$i$-th cohomology group of the graded cochain complex \eqref{eq:abfp} is isomorphic to~$H^{cw}_{d-i}(\Pp;H^{*}(BS))$ (as a~$\Q$-vector space) for any~$i\geq 0$, where~$d=\dim \Pp=k-b$.
	\end{lemma}
		
	\begin{theorem}\label{thm:cmdiag}
		Let~$D$ be any strongly reduced~$T$-diagram over a CW poset~$\Pp$.
		Then the~$T$-action on~$\hoc D$ is Cohen-Macaulay iff the sheaf~$H^{*}(BS)$ is acyclic.
		In this case there are the isomorphisms
		\[
		H^{cw}_{d-i}(\Pp;H^{*}(BS))\cong
		\begin{cases}
			H_{T}^{*}(\hoc D),\ i=0,\\
			0,\ \mbox{otherwise}.
		\end{cases}
		\]
		\[
		\lm^{t-i} H^{t}(D)\cong
		\begin{cases}
			\Tor^{-i,2t}_{H^{*}(BT)}(H^{*}_{T}(\hoc D),\Q),\ i\leq t,\\
			0,\ \mbox{otherwise}.
		\end{cases}
		\]
	\end{theorem}
	\begin{proof}
		MVSS collapse for~$BS$ implies that the group~$H^{i}_{T}(\hoc D)$ is isomorphic to the sum of respective cohomology of the sheaf~$H^{*}(BS)$ (Theorem~\ref{thm:btoric}, Proposition~\ref{pr:borfib}).
		The latter cohomology are isomorphic to cellular cohomology by Theorem~\ref{thm:evtu}.
		(Here, notice that the natural isomorphism of $\Q$-modules implies isomorphism of $H^{*}(BT;\Q)$-modules automatically.)
		Therefore, the equivalence follows by Lemma~\ref{lm:compspec}, and the first formula holds.
		(Furthermore, the comparison SS collapses.)
		The second formula follows directly by the abutment of the comparison SS (Corollary~\ref{cor:acyc}).
	\end{proof}
	
	\begin{remark}\label{rem:orpos}
		In the conditions of Lemma~\ref{lm:compspec} we find by taking degree~$0$ component of the coefficient sheaf~$H^{*}(BS)$ (compare with Remark~\ref{rem:acycnull})
		\[
		H^{cw}_{d-i}(\Pp;\kappa\Q)=
		H_{cw}^{i}(\Pp^{op};\kappa\Q)=
		H^{i}(|\ord \Pp|;\Q)=0,\ i\neq 0.
		\]
	\end{remark}
	
	\begin{remark}\label{rem:efhoc}
		A~$T$-diagram~$D$ such that~$\hoc D$ is equivariantly formal has a fixed point of the~$T$-action.
		One can easily deduce that equivariant formality condition for~$\hoc D$ is equivalent to~$H^{odd}(\hoc D;\Q)=0$.
		Let~$D$ be such that~$\hoc D$ is equivariantly formal.
		This implies that~$H^{*}(BS)$ is an acyclic sheaf by Theorem~\ref{thm:cmdiag}.
        Then the vanishing of higher~$\Tor$-groups (by equivariant formality of~$\hoc D$) imply the formula (where on the right one has the direct sum components of~$H^{*}(\hoc D;\Q)$)
		\begin{equation}\label{eq:diagnum}
			\lm^{i} H^{j}(D)\cong
			\begin{cases}
				\Tor^{0,2j}_{H^{*}(BT)}(H^{*}(\hoc BS),\Q),\ i=j,\\
				0,\ \mbox{otherwise}.
			\end{cases}
		\end{equation}
	\end{remark}
	
	\section{Skeletons of compact nonsingular toric varieties}\label{sec:betti}
	
	Let~$D$ be any strongly reduced~$T$-diagram ($T=(S^{1})^{k}$) over any CW poset~$\Pp$.
	Suppose that~$\hoc D$ is equivariantly formal over~$\Q$ (hence,~$H^{*}(BS)$ is acyclic, and~$k=d$, where~$d:=\dim \Pp$).
	Then by Remark~\ref{rem:efhoc} the groups~$H^{odd}(\hoc D;\Q)=0$ are zero.
	Under this assumption, we compute the Betti numbers of~$\hoc D$ in terms of the respective~$h$-numbers (Proposition~\ref{pr:bettief}) and bigraded Betti numbers for all skeletons of $D$ (Corollary~\ref{cor:efbb}).
	We show that our formula agrees with the that for the skeletons of~$\C P^{m-1}$ following from~\cite{fu-19} and~\cite{gr-th-07}.
	
	\subsection{Betti numbers for~$T$-diagrams in equivariantly formal case}
	
	Recall that the~$h$-numbers of a CW poset~$\Pp$,~$\dim \Pp=d$, are given by
	\[
	h_{i}(\Pp):=\sum_{j=0}^{d} (-1)^{i-j} f_{j-1}(\Pp)\binom{d-j}{i-j},\ i=0,\dots,d,
	\]
	where~$f_{j}(\Pp)$ denotes the number of rank~$j+1$ elements in~$\Pp$ (in particular,~$f_{-1}(\Pp)=1$).
	
	\begin{proposition}\label{pr:bettief}
		For any equivariantly formal~$\hoc D$, where~$D$ is a~$T$-diagram over any CW poset~$\Pp$,~$\dim \Pp=d$, one has
		\begin{equation}\label{eq:betti1}
			b_{2i}(\hoc D)=h_{d-i}(\Pp), i=0,\dots,d.
		\end{equation}
	\end{proposition}
	\begin{proof}
		Consider the cochain complex
		\[
		(C^{*},d)=(C^{i}_{cw}(\Pp^{op}, H^{*}(D)), d)=(C_{i}^{cw}(\Pp, H^{*}(D)), d),
		\] 
		with the standard differential.
		By the standard property of the Euler characteristic~$\chi$ of a cochain complex, one has
		\[
		\chi(C^{*}_{cw}(\Pp^{op}; H^{*}(D)))=
        \chi(H^{*}_{cw}(\Pp^{op}; H^{*}(D))).
		\]
		By the definition of the cochain complex and simplicial homology of the sphere~$\Pp_{< p}$, one has
		\[
			\chi(C^{*}_{cw}(\Pp^{op}, H^{i}(D)))=
            (-1)^{d}\chi(C_{*}^{cw}(\Pp, H^{i}(D)))=
			(-1)^{d}\sum_{j=0}^{d} (-1)^{j-1} f_{j-1}(\Pp)\binom{d-j}{i}.
		\]
		In the last equality we use~$\rk H^{i}(D(p))=\binom{d-j}{i}$,~$j=\rk p$,~$p\in \Pp$, by strong reducedness of~$D$.
		By formula \eqref{eq:diagnum} of Remark~\ref{rem:efhoc},
		\[
		\widetilde{\chi}(H^{i}_{cw}(\Pp^{op}; H^*(D)))=(-1)^{i+1} \cdot b_{2i}(\hoc D).
		\]
		The identity \eqref{eq:betti1} now follows directly from the last three identities.
	\end{proof}
	
	\begin{example}
		Let~$\Sigma$ be the fan of any compact nonsingular toric variety~$X$ of dimension~$d$.
		Then one has~$h_{d-i}(\Sigma)=h_{i}(\Sigma)$, and the above equality takes a standard form.
	\end{example}
	
	\begin{remark}
		We remark that the previous example generalizes to the case when~$\Pp$ is both a CW poset and an orientable homology manifold.
		In this case,~$\Pp$ satisfies Dehn-Sommerville relations~$h_{d-i}(\Pp)=h_{i}(\Pp)$ by Poincare-Lefschetz duality, e.g. see~\cite{ch-23},~\cite{ay-17}.
		For example, this is the case for any compact nonsingular toric variety.
	\end{remark}
	
	\subsection{Betti numbers for skeletons of~$T$-diagrams in equivariantly formal case}

	\begin{definition}
		For any CW poset~$\Pp$ let
		\[
		\iota_{q}\colon \sk^{q} \Pp=\lb p\in \Pp\colon \rk p\leq q\rb\to \Pp,\ 
		\tau_{q}\colon \cosk^{q} \Pp=\lb p\in \Pp\colon \rk p\geq q\rb\to \Pp,
		\]
		be the natural embedding of the \textit{$q$-skeleton} (\textit{$q$-coskeleton}, respectively) of~$\Pp$.
		Let~$D$ be a~$T$-diagram over~$\Pp$.
		Put~$D_{\leq q}=\iota_{q}^{*} D$ to be the diagram over~$\Pp_{\leq q}$.
		We call~$\hoc D_{\leq q}$,~$\hoc D_{\geq q}$ a \textit{$q$-skeleton} (\textit{$q$-coskeleton}, respectively) of~$\hoc D$.
	\end{definition}
	
	By Theorem~\ref{thm:cmdiag}, any skeleton of a $T$-diagram with equivariantly formal $\hoc D$ is Cohen-Macaulay.
	
	\begin{remark}
		One can check that the~$q$-skeleton~$X_{\sk^{q}\Sigma}$ of any toric variety~$X_{\Sigma}$ with~$T$-action is~$T$-equivariantly homotopy equivalent to~$\hoc (D_{\Sigma})_{\leq q}$, where~$\sk^{q}\Sigma$ denotes the fan consisting of all cones of dimension~$\leq q$ from~$\Sigma$.
		The subvariety~$X_{\sk^{q}\Sigma}$ has the natural open embedding to~$X_{\Sigma}$.
		The complement of its image is the closed subvariety consisting of all orbits of dimension~$\leq n-q-1$, where~$n=\dim T$.
		This complement is~$T$-equivariantly homotopy equivalent to~$\hoc (D_{\Sigma})_{\geq n-q-1}$ by a similar argument to~\cite{fr-10}.
	\end{remark}
		
	\begin{lemma}\label{lm:bskel}
		For any toric diagram~$D$ over~$\Pp$ one has
		\begin{equation}\label{eq:cohdes}
			\lmm^{i} H^{j}(D)_{\leq q}\cong
			\begin{cases}
				\lm^{i} H^{j}(D),\ i<q\\
				0,\ i>q.
			\end{cases}
		\end{equation}
	\end{lemma}
	\begin{proof}
		Notice that the skeleton of a CW poset is a CW poset.
		By Theorem~\ref{thm:evtu}, the sheaf cohomology from the claim is isomorphic to the cellular cohomology.
		Then the claim follows directly by the properties of cellular cohomology. 
	\end{proof}
	
	One can find the dimensions of the groups~$\lmm^{q} H^{j}(D)_{\leq q}$ using reduced Euler characteristic~$\widetilde{\chi}$ and MVSS collapse (Theorem~\ref{thm:rcollf}, or the orbit spectral sequence collapse for smooth fans~\cite{fr-06}, since a skeleton of a smooth fan is smooth).
	We give the resulting formula in the particular case below (the proof is straightforward and is omitted).
	
	\begin{corollary}\label{cor:efbb}
		Let~$D$ be any strongly reduced toric~$T$-diagram such that~$\hoc D$ is equivariantly formal.
		Then one has
		\[
		\beta^{0,2j}(D_{\leq q})=\beta^{0,2j}(D),\ j<q,
		\]
		\[
		\beta^{-i,2(i+q)}(D_{\leq q})=
		(-1)^{q+1}\sum_{p\in \sk^{q} \Pp} \binom{d-\rk p}{i+q}\widetilde{\chi}((\sk^{q} \Pp)_{> p}).
		\]
		All the remaining bigraded Betti numbers of~$D_{\leq q}$ are zero.
	\end{corollary}
	
	\begin{remark}
		The formulas of Corollary~\ref{cor:efbb} give the Betti numbers of skeletons for any toric manifold and any compact nonsingular toric variety, because one has
		\[
		\beta^{q}(\hoc D_{\leq q})=\sum_{-i+2j=q}\beta^{-i,2j}(D_{\leq q}),
		\]
		where~$D$ is the corresponding diagram to such a space, as explained for the EMSS (which is well known to abut over~$\Q$ at its page~$2$) above.
	\end{remark}
	
	\subsection{Skeletons of the complex projective space}
	
	We conclude this section with a comparison of Betti numbers formulas in the case of skeletons for~$\C P^{m-1}$ implied by~\cite{gr-th-07, fu-19} versus the above.
	
	\begin{example}
		Let~$\Pp=\Cone \partial\Delta_{m}$,~$T=T^m/S^{1}_{d}\cong T^{m-1}$, where~$S^{1}_{d}$ is the big diagonal circle, and let~$D$ be the~$T$-diagram
		\[
		D(I)=T^{m}/(T^{I}\cdot S^{1}_{d}),
		\]
		$T^I$ denotes the product in~$T^m=(S^1)^m$ of coordinate circles~$S^1_{i}$,~$i\in I$, and the arrows of~$D$ are induced by the embedding of~$T^{I}\cdot S^{1}_{d}\to T^{J}\cdot S^{1}_{d}$.
		One has~$\C P^{m-1}\cong \hoc D$~\cite{we-zi-zi-99}.
		One has
		\[
		\dim \lmm ^{i} H^{j}(D)=
		\begin{cases}
			1, \mbox{ if } i=j\leq m-1,\\
			0,\ \mbox{otherwise}.
		\end{cases}
		\]
		Consider~$\widetilde{\Pp}=\sk^{q+1} \Pp=\Cone \Delta^{q}_{m}$ (the cone over the~$q$-dimensional skeleton of the simplex~$\Delta_{m}$ on~$m$ vertices).
		Then
		\[
		\dim \lmm ^{i} H^{j}(D)_{\leq q+1}=
		\begin{cases}
			1,\ i=j\leq q,\\
			0,\ i,j\leq q, i\neq j.
		\end{cases}
		\]
		Furthermore,
		\[
		\dim \lmm ^{i} H^{j}(D)_{\leq q+1}=0,\ i>j.
		\]
		It remains to compute dimensions for~$i=q+1$ and~$j\geq q+1$.
		The number of~$i$-dimensional faces in~$\widetilde{C}$ is equal to~$\binom{m}{i+1}$,~$-1\leq i\leq q$.
		Any such face~$I$,~$\dim I=i$, has the link
		\[
		\lk_{\widetilde{\Pp}} I\cong\Delta_{m-1-i}^{q-1-i}.
		\]
		One has
		\[
		\Delta_{m}^{l}\simeq (S^{l})^{\vee \binom{m-1}{l+1}}.
		\]
		Therefore,
		\[
		\widetilde{\chi}(\lk_{\widetilde{\Pp}} I)=(-1)^{q-1-i}\binom{m-2-i}{q-i}.
		\]
		We compute
		\begin{equation}\label{eq:1mon}
			\dim \lmm^{q+1} H^{j}(D)_{\leq q+1}=
			(-1)^{q}\sum_{i=-1}^{q} \binom{m}{i+1} (-1)^{q-1-i} \binom{m-2-i}{q-i} \binom{m-2-i}{j}.
		\end{equation}
		On the other hand, by~\cite{gr-th-07, fu-19} it follows that
		\[
		\hoc D\simeq \C P^{q+1}\vee \biggl(\bigvee_{i=0}^{q}\bigvee_{b=q-i+2}^{m-1-i} (S^{q+b+i+1})^{\vee\binom{m-1-i}{b}\binom{b-1}{q-i-1}}\biggr)\vee
		\bigvee_{i=1}^{m-q-2} (S^{2q+2+i})^{\vee\binom{m-q-2}{i}}.
		\]
		This implies by a routine computation that
		\begin{equation}\label{eq:2mon}
			b_{q+1+j}(\hoc D)=\dim \lmm ^{q+1} H^{j}(D)_{\leq q+1}=
			\sum_{i=0}^{q+1}\binom{m-1-i}{j-i}\binom{j-i-1}{q-i+1}.
		\end{equation}
		We compare \eqref{eq:1mon} with \eqref{eq:2mon} as follows.
		The expressions \eqref{eq:1mon}, \eqref{eq:2mon} are identified respectively with
		\[
		\binom{m-1}{j}\binom{m-1}{q+1} \pFq{3}{2}{-q-1,-m,j-m+1}{1-m,1-m}{1},
		\]
		\[
		\binom{m-1}{j}\binom{j-1}{q+1} \pFq{3}{2}{-q-1,1,-j}{1-m,1-j}{1},
		\]
		where
		\[
		\pFq{3}{2}{a,b,c}{d,e}{1}:=\sum_{i=0}^{\infty} \frac{1}{i!}\frac{(a)_{i}(b)_{i}(c)_{i}}{(d)_{i}(e)_{i}},\ a,b,c,d,e\in\Z,
		\]
		denotes the value of the generalized hypergeometric function at unit (if the series converge), and
		\[
		(d)_{n}=d\cdot (d+1)\cdots (d+n-1),
		\]
		denotes the Pochhammer symbol.
		As expected, we obtain the equality of these two expressions directly from the identity~\cite[Appendix,~$(II)$]{sr-je-ra-92}:
		\[
		\pFq{3}{2}{a,b,c}{d,e}{1}=
		\frac{(d+e-b-c)_{n}}{(d)_{n}}\pFq{3}{2}{-n,e-b,e-c}{e,d+e-b-c}{1},\ n\in \N,
		\]
		by substituting
		\[
		n=q+1,\ b=1,\ c=-j,\ e=1-m,\ d=1-j.
		\]
	\end{example}
		

\end{document}